\documentclass[reqno]{amsart}
\usepackage{amsmath, amsthm, amssymb, amstext}

\usepackage[left=3.5cm,right=3.5cm,top=1.5cm,bottom=1.5cm,includeheadfoot]{geometry}
\usepackage{hyperref,xcolor}
\hypersetup{
 pdfborder={0 0 0},
 colorlinks,
}
\usepackage{enumitem}
\allowdisplaybreaks

\newtheorem{theorem}{Theorem}

\newtheorem{proposition}[theorem]{Proposition}

\DeclareMathOperator*{\divergenz}{div}              %
\DeclareMathOperator*{\loc}{loc}         %

\newcommand{\N}{\mathbb{N}}

\newcommand{\R}{\mathbb{R}}

\newcommand{\RN}{\mathbb{R}^N}

\newcommand{\diff}{\mathrm{d}}

\newcommand{\Lp}[1]{L^{#1}(\RN)}

\newcommand{\Wp}[1]{W^{1,#1}(\RN)}

\newcommand{\eps}{\varepsilon}
\newcommand{\ph}{\varphi}

\newcommand{\into}{\int_{\mathbb{R}^N}}
\newcommand{\weak}{\rightharpoonup}

\newcommand{\close}{\overline{\Omega}}

\newcommand{\cprime}{$'$}

\renewcommand{\l}{\left}
\renewcommand{\r}{\right}

\newcommand{\WH}{W^{1, \mathcal{H}}(\RN)}
\numberwithin{theorem}{section}
\numberwithin{equation}{section}

\def\ge{\geqslant}

\def\phi{\varphi}

\def\ykh#1{\left(#1\right)}

\title[Combined effects of singular and superlinear nonlinearities]{Combined effects of singular and superlinear nonlinearities in singular double phase problems in $\mathbb{R}^N$}

\author[W.\,Liu]{Wulong Liu}
\address[W.\,Liu]{School of Science, Jiangxi University of Science and Technology, Ganzhou, Jiangxi 341000, PR China}
\email{liuwul000@gmail.com}

\author[P.\,Winkert]{Patrick Winkert}
\address[P.\,Winkert]{Technische Universit\"{a}t Berlin, Institut f\"{u}r Mathematik, Stra\ss e des 17.\,Juni 136, 10623 Berlin, Germany}
\email{winkert@math.tu-berlin.de}



\begin{document}

\begin{abstract}
	This paper is concerned with multiplicity results for parametric singular double phase problems in $\mathbb{R}^N$ via the Nehari manifold approach. It is shown that the problem under consideration has at least two nontrivial weak solutions provided the parameter is sufficiently small. The idea is to split the Nehari manifold into three disjoint parts minimizing the energy functional on two of them. The third set turns out to be the empty set for small values of the parameter.
\end{abstract}

\maketitle

\section{Introduction}

In this paper we consider quasilinear elliptic equations in $\R^N$ driven by the double phase operator and a right-hand side consisting of a singular and a parametric term. Recall that the double phase operator is defined by
\begin{align}\label{operator_double_phase}
	\divergenz(A(u)):=\divergenz \Big(|\nabla u|^{p-2} \nabla u+ \mu(x) |\nabla u|^{q-2} \nabla u\Big)\quad \text{for }u\in \WH,
\end{align}
where $1<p<q$ and with a nonnegative weight function $\mu\in\Lp{1}$ while  $\WH$ is the appropriate Musielak-Orlicz Sobolev function space, see Section \ref{section_2} for more details. Note that the operator in \eqref{operator_double_phase} is closely related to the $p$-Laplacian (for $\mu\equiv 0$) and the $(q,p)$-Laplacian (for $\inf_{\close} \mu>0$).  The physical motivation  of such operator goes back to Zhikov \cite{Zhikov-1986} who introduced the functional
\begin{align}\label{integral_minimizer}
	J(u)=\int_\Omega \Big(|\nabla  u|^p+\mu(x)|\nabla  u|^q\Big)\,\diff x
\end{align}
in order to provide models for strongly anisotropic materials where $\Omega$ is a bounded domain. Other applications can be found in the works of Bahrouni-R\u{a}dulescu-Repov\v{s} \cite{Bahrouni-Radulescu-Repovs-2019} on transonic flows, Benci-D'Avenia-Fortunato-Pisani \cite{Benci-DAvenia-Fortunato-Pisani-2000} on quantum physics and  Cherfils-Il\cprime yasov \cite{Cherfils-Ilyasov-2005} on reaction diffusion systems.

It is easy to see that $J(\cdot)$ defined in \eqref{integral_minimizer} has unbalanced growth and the ellipticity of $J(\cdot)$ changes on the set where the weight function is zero. For further studies of functionals with $(p,q)$-growth we refer to the works of Baroni-Colombo-Mingione \cite{Baroni-Colombo-Mingione-2015,Baroni-Colombo-Mingione-2018}, Byun-Oh \cite{Byun-Oh-2020}, Colombo-Mingione \cite{Colombo-Mingione-2015a,Colombo-Mingione-2015b}, De Filippis-Palatucci \cite{De-Filippis-Palatucci-2019},
Marcellini \cite{Marcellini-1991,Marcellini-1989b}, Ok \cite{Ok-2018,Ok-2020}, Ragusa-Tachikawa \cite{Ragusa-Tachikawa-2020} and the references therein.

In this paper we study the following singular problem
\begin{equation}\label{problem}
	\begin{aligned}
		&-\divergenz(A(u))+u^{p-1}+\mu(x)u^{q-1}= \xi(x)u^{-\alpha}+\lambda u^{\nu-1}\quad \text{in } \R^N,
	\end{aligned}
\end{equation}
with $\lambda>0$ being the parameter to be specified where we assume the following assumptions on the data:
\begin{enumerate}
	\item[\textnormal{(H)}:]
	\begin{enumerate}
		\item[\textnormal{(i)}]
		$1<p<q<N$, $ q<p^*=\frac{Np}{N-p}$ and $0 \leq \mu(\cdot)\in L^\infty(\R^N)$;
		\item[\textnormal{(ii)}]
		$0<\alpha<1$ and $q<\nu<p^*$;
		\item[\textnormal{(iii)}]
		$\xi(x) > 0$ for a.\,a.\,$x\in\R^N$ and $\xi \in \Lp{\sigma}\cap  L^\infty(\R^N)$, where $\sigma>1$ and there exists $\kappa\geq 1$ such that
		\begin{align*}
			p < \kappa \leq p^*\quad\text{and}\quad \frac{1}{\sigma}+\frac{1-\alpha}{\kappa}=1.
		\end{align*}
	\end{enumerate}
\end{enumerate}

We call a function $u \in\WH$ a weak solution of problem \eqref{problem} if $\xi(\cdot)u^{-\alpha} v \in \Lp{1}$, $u> 0$ for a.\,a.\,$x\in\R^N$ and if
\begin{equation}\label{weak_solution}
	\begin{split}
		& \into \Big(|\nabla u|^{p-2} \nabla u+ \mu(x) |\nabla u|^{q-2} \nabla u\Big) \cdot \nabla v \,\diff x+\into \Big(u^{p-1}+\mu(x)u^{q-1}\Big)v \,\diff x \\
		&= \into \xi(x)u^{-\alpha} v \,\diff x+\lambda \into u^{\nu-1}v\,\diff x
	\end{split}
\end{equation}
is satisfied for all $v \in \WH$. It is easy to see that the definition of a weak solution in \eqref{weak_solution} is well-defined. The corresponding energy functional  $\Psi_\lambda\colon\WH\to \R$  for problem \eqref{problem} is given by
\begin{equation*}
	\begin{split}
		\Psi_\lambda(u)&=\frac{1}{p} \|u\|_{1,p} ^p+\frac{1}{q}\Big(\|\nabla u\|_{q,\mu}^q+\|u\|_{q,\mu}^q \Big)-\frac{1}{1-\alpha}\into \xi(x)|u|^{1-\alpha}\,\diff x-\frac{\lambda}{\nu}\|u\|_{\nu}^{\nu}.
	\end{split}
\end{equation*}

Our main result reads as follows.

\begin{theorem}\label{main_result}
	Let hypotheses \textnormal{(H)} be satisfied. Then there exists $\lambda^*>0$ such that for all $\lambda \in (0,\lambda^*]$ problem \eqref{problem} has at least two weak solutions $u^*, v^* \in \WH$ such that $\Psi_\lambda(u^*)<0<\Psi_\lambda(v^*)$.
\end{theorem}

The idea in the proof of Theorem \ref{main_result} is the usage of the so-called Nehari manifold which is defined by
\begin{align*}
	\mathcal{N}_\lambda=\left\{u\in\WH\setminus\{0\}\,:\,
	\|u\|_{1,p} ^p+\|\nabla u\|_{q,\mu}^q+\|u\|_{q,\mu}^q  =\into \xi(x)|u|^{1-\alpha}\,\diff x+\lambda \|u\|_{\nu}^{\nu}\right\}.
\end{align*}
We point out that $\mathcal{N}_\lambda$ contains all nontrivial weak solutions of problem \eqref{problem} but $\mathcal{N}_\lambda$ is smaller than the whole space $\WH$. We are going to split $\mathcal{N}_\lambda$ into three disjoint parts and the minimizers of two of them provide us the two different weak solutions. The third set turns out to be the empty set.

This method was introduced by Nehari \cite{Nehari-1961, Nehari-1960} and  turned into a very powerful tool in order to find solutions for differential equations, especially when no regularity theory exists as it is the case for our problem. We refer to  Szulkin-Weth \cite{Szulkin-Weth-2010} for a nice introduction to this technique.

To the best of our knowledge this is the first work dealing with a singular double phase problem on the whole of $\R^N$. The motivation of the paper comes from a recent work of Liu-Dai-Papageorgiou-Winkert \cite{Liu-Dai-Papageorgiou-Winkert-2021} who studied a similar problem but on bounded domains $\Omega\subseteq\R^N $ and solutions in the space $W^{1,\mathcal{H}}_0(\Omega)$. In our setting, the situation is different because we do not have simple embeddings of the form $\Lp{s_1} \hookrightarrow \Lp{s_2}$ for $s_2 \leq s_1$ due to the unboundedness of the domain. Instead we need to strengthen the hypotheses on $\xi(\cdot)$, see \textnormal{(H)(iii)}. On the other hand, since we do not use an equivalent norm as it is the case in \cite{Liu-Dai-Papageorgiou-Winkert-2021}, we were able to relax the assumptions on the weight function $\mu(\cdot)$ and also on the exponents $p$ and $q$. Indeed, in \cite{Liu-Dai-Papageorgiou-Winkert-2021} it is supposed that $\mu \in C^{0,1}(\close)$ while we only need $\mu$ to be bounded. Moreover, the condition
\begin{align}\label{condition_equivalent_norm}
	\frac{q}{p}<1+\frac{1}{N}
\end{align}
is required in \cite{Liu-Dai-Papageorgiou-Winkert-2021} which was used for the first time by Baroni-Colombo-Mingione \cite[see (1.8)]{Baroni-Colombo-Mingione-2015} in order to obtain regularity results of local minimizers for double phase integrals. The condition is needed, for instance, for the density of smooth functions in suitable Musielak-Orlicz Sobolev spaces. In our paper it is enough to suppose that $q<p^*$ which is weaker than inequality \eqref{condition_equivalent_norm}.

Existence results for singular double phase problems on bounded domains have only been done in few recent works. Chen-Ge-Wen-Cao \cite{Chen-Ge-Wen-Cao-2020} treated the problem
\begin{equation*}
	\begin{aligned}
		&-\divergenz \Big(|\nabla u|^{p-2} \nabla u+ \mu(x) |\nabla u|^{q-2} \nabla u\Big)= b(x)|u|^{-\vartheta-1}u+\lambda f(x,u)\quad && \text{in } \Omega,
		&u=0&& \text{on }\partial\Omega,
	\end{aligned}
\end{equation*}
where $0<\vartheta<1$ and  $f$ is a Carath\'{e}odory function. It is shown that the problem has at least one weak solution with negative energy. Singular Finsler double phase problems have been recently studied by Farkas-Winkert \cite{Farkas-Winkert-2021} where the existence of a weak solution to the problem
\begin{equation*}
		-\divergenz\left(F^{p-1}(\nabla u)\nabla F(\nabla u) +\mu(x) F^{q-1}(\nabla u)\nabla F(\nabla u)\right) = u^{p^{*}-1}+\lambda \l(u^{\alpha-1}+g(u)\r)\quad\text{in }\Omega,
\end{equation*}
with $u=0$ on $\partial\Omega$ is shown by applying variational tools, while $(\R^N,F)$ stands for a Minkowski space. This work was extended to singular Finsler problems with nonlinear boundary condition by  Farkas-Fiscella-Winkert \cite{Farkas-Fiscella-Winkert-2021}. In all these works, the domain is bounded, only the existence of one weak solution is shown and the methods are different from ours.

Finally, we mention some papers dealing with existence results for a double phase setting (or nonstandard growth) on bounded domains without singular term by applying different tools like critical point theory, Moses theory and variational tools as comparison and truncation techniques. We refer to Bahrouni-R\u{a}dulescu-Winkert \cite{Bahrouni-Radulescu-Winkert-2020}, Barletta-Tornatore \cite{Barletta-Tornatore-2021}, Colasuonno-Squassina \cite{Colasuonno-Squassina-2016}, El Manouni-Marino-Winkert \cite{El-Manouni-Marino-Winkert-2022}, Gasi\'nski-Papa\-georgiou \cite{Gasinski-Papageorgiou-2019}, Gasi\'nski-Winkert \cite{Gasinski-Winkert-2020a,Gasinski-Winkert-2020b,Gasinski-Winkert-2021}, Liu-Dai \cite{Liu-Dai-2018,Liu-Dai-2018b}, Marino-Winkert \cite{Marino-Winkert-2020}, Papageorgiou-R\u{a}dulescu-Repov\v{s} \cite{Papageorgiou-Radulescu-Repovs-2020d}, Papa\-georgiou-Vetro-Vetro \cite{Papageorgiou-Vetro-Vetro-2020}, Perera-Squassina \cite{Perera-Squassina-2019}, R\u{a}dulescu \cite{Radulescu-2019}, Zeng-Bai-Gasi\'nski-Winkert \cite{Zeng-Bai-Gasinski-Winkert-2020, Zeng-Gasinski-Winkert-Bai-2021} and the references there\-in. Only the work of Liu-Dai \cite{Liu-Dai-2020} deals with a double phase operator on the whole of $\R^N$ but without a singular term.

The paper is organized as follows. In Section \ref{section_2} we present the main properties of the Musielak-Orlicz Sobolev spaces defined on $\R^N$ including some embedding results. Section \ref{section_3} is devoted to the proof of Theorem \ref{main_result} which is splited into several propositions.

\section{Preliminaries }\label{section_2}

This section is concerned to the main properties of Musielak-Orlicz Sobolev spaces defined on $\R^N$. The results are mainly taken from Colasuonno-Squassina \cite{Colasuonno-Squassina-2016}, Harjulehto-H\"{a}st\"{o} \cite{Harjulehto-Hasto-2019}, Musielak \cite{Musielak-1983} and Liu-Dai \cite{Liu-Dai-2020}.

For  $1\leq r<\infty$, $\Lp{r}$ and $L^r(\R^N;\R^ N)$ stand for the usual Lebesgue spaces endowed with the norm $\|\cdot\|_r$. Furthermore, we denote by $\Wp{r}$ ($1<r<\infty$)  the usual Sobolev space equipped with the norm
\begin{align*}
	\|u\|_{1,r}=\big(\|\nabla u\|_{r}^{r}+\|u\|_r^ r\big)^\frac{1}{r}.
\end{align*}

Suppose condition \textnormal{(H)(i)} and let $M(\R^N)$ be the space of all measurable functions $u\colon \R^ N\to\R$. Moreover, let $\mathcal{H}\colon \RN \times [0,\infty)\to [0,\infty)$ be the function defined by
\begin{align*}
	\mathcal H(x,t)= t^p+\mu(x)t^q.
\end{align*}
Then, by $L^\mathcal{H}(\RN)$ we denote the Musielak-Orlicz Lebesgue space given by
\begin{align*}
	L^{\mathcal{H}}(\R^N)=\left \{u \in M(\R^N) \,:\,\into \mathcal{H} (x,|u|)\,\diff x < +\infty \right \},
\end{align*}
which is endowed with the Luxemburg norm
\begin{align*}
	\|u\|_{\mathcal{H}} = \inf \left \{ \tau >0 \,:\, \into \mathcal{H} \l(x,\frac{|u|}{\tau}\r)\,\diff x\leq 1  \right \}.
\end{align*}
From Liu-Dai \cite[Theorem 2.7 \textnormal{(i)}]{Liu-Dai-2020} we know that the space $L^\mathcal{H}(\R^N)$ is a reflexive Banach space.

Furthermore, we introduce the seminormed space
\begin{align*}
	L^q_\mu(\R^N)=\left \{u \in M(\R^N) \,:\,\into \mu(x) |u|^q \,\diff x< +\infty \right \}
\end{align*}
with the seminorm
\begin{align*}
	\|u\|_{q,\mu} = \left(\into \mu(x) |u|^q \,\diff x \right)^{\frac{1}{q}}.
\end{align*}
Similarly, we define $L^q_\mu(\R^N;\R^N)$.

The Musielak-Orlicz Sobolev space $W^{1,\mathcal{H}}(\RN)$ is defined by
\begin{align*}
	W^{1,\mathcal{H}}(\RN)= \Big \{u \in L^\mathcal{H}(\RN) \,:\, |\nabla u| \in L^{\mathcal{H}}(\RN) \Big\}
\end{align*}
equipped with the norm
\begin{align*}
	\|u\|= \|\nabla u \|_{\mathcal{H}}+\|u\|_{\mathcal{H}},
\end{align*}
where $\|\nabla u\|_\mathcal{H}=\|\,|\nabla u|\,\|_{\mathcal{H}}$. As before, it is clear that $\WH$ is a reflexive Banach space, see Liu-Dai \cite[Theorem 2.7 \textnormal{(ii)}]{Liu-Dai-2020}.

Moreover, we have the continuous embedding
\begin{align}\label{embedding}
	\WH \hookrightarrow \Lp{r}\quad \text{for all }r \in [p,p^*],
\end{align}
see Liu-Dai \cite[Theorem 2.7(iii)]{Liu-Dai-2020}.

Let
\begin{align}\label{modular}
	\varrho(u)=\into \Big(|\nabla u|^p+\mu(x)|\nabla u|^q+|u|^p+\mu(x)|u|^q\Big)\,\diff x.
\end{align}
The norm $\|\cdot\|$ and the modular function $\varrho$ are related as follows, see Liu-Dai \cite[Proposition 2.1]{Liu-Dai-2018} or Crespo-Blanco-Gasi\'nski-Harjulehto-Winkert \cite[Proposition 2.15]{Crespo-Blanco-Gasinski-Harjulehto-Winkert-2021}.

\begin{proposition}\label{proposition_modular_properties}
	Let \textnormal{(H)(i)} be satisfied, let $y\in \Wp{\mathcal{H}}$ and let $\varrho$ be defined by \eqref{modular}. Then the following hold:
	\begin{enumerate}
		\item[\textnormal{(i)}]
			If $y\neq 0$, then $\|y\|=\lambda$ if and only if $ \varrho(\frac{y}{\lambda})=1$;
		\item[\textnormal{(ii)}]
		$\|y\|<1$ (resp.\,$>1$, $=1$) if and only if $ \varrho(y)<1$ (resp.\,$>1$, $=1$);
		\item[\textnormal{(iii)}]
		If $\|y\|<1$, then $\|y\|^q\leq \varrho(y)\leq\|y\|^p$;
		\item[\textnormal{(iv)}]
		If $\|y\|>1$, then $\|y\|^p\leq \varrho(y)\leq\|y\|^q$;
		\item[\textnormal{(v)}]
		$\|y\|\to 0$ if and only if $ \varrho(y)\to 0$;
		\item[\textnormal{(vi)}]
		$\|y\|\to +\infty$ if and only if $ \varrho(y)\to +\infty$.
	\end{enumerate}
\end{proposition}

\section{Proof of Theorem \ref{main_result}}\label{section_3}

We recall that the corresponding energy functional $\Psi_\lambda\colon\WH\to \R$ of problem \eqref{problem} is defined by
\begin{equation*}
	\begin{split}
		\Psi_\lambda(u)&=\frac{1}{p} \|u\|_{1,p} ^p+\frac{1}{q}\big(\|\nabla u\|_{q,\mu}^q+\|u\|_{q,\mu}^q \big)-\frac{1}{1-\alpha}\into \xi(x)|u|^{1-\alpha}\,\diff x-\frac{\lambda}{\nu}\|u\|_{\nu}^{\nu}.
	\end{split}
\end{equation*}
We know that $\Psi_\lambda$ is not a $C^1$-functional because of the singular term in problem \eqref{problem}. For $u \in \WH\setminus\{0\}$, we introduce the fibering function $\omega_u\colon[0,+\infty)\to \R$ given by 
\begin{align*}
	\omega_u(t)=\Psi_\lambda (tu)\quad\text{for all }t\geq 0.
\end{align*}
Clearly, $\omega_u \in C^\infty((0,\infty))$. The idea in the proof of Theorem \ref{main_result} is to minimize the energy functional $\Psi_\lambda\colon\WH\to \R$ on the so-called Nehari manifold $\mathcal{N}_\lambda$. This manifold is defined as
\begin{align*}
	\mathcal{N}_\lambda
	&=\bigg\{u\in\WH\setminus\{0\}\,:\,\|u\|_{1,p} ^p+\|\nabla u\|_{q,\mu}^q+\|u\|_{q,\mu}^q  =\into \xi(x)|u|^{1-\alpha}\,\diff x+\lambda \|u\|_{\nu}^{\nu}\bigg\}\\
	 & =\l\{u\in\WH\setminus\{0\}\,:\, \omega_u'(1)=0\r\}.
\end{align*}
In order to find our two different weak solutions stated in Theorem \ref{main_result} we have to decompose the set $\mathcal{N}_\lambda$ into three disjoint sets. To this end, let
\begin{align*}
	\mathcal{N}_\lambda^+
	&=\bigg\{u \in \mathcal{N}_\lambda\,:\, (p+\alpha-1)\|u\|_{1,p} ^p+(q+\alpha-1)\big(\|\nabla u\|_{q,\mu}^q+\|u\|_{q,\mu}^q \big)>\lambda (\nu+\alpha-1) \|u\|_{\nu}^{\nu}\bigg\},\\
	& =\l\{u\in\mathcal{N}_\lambda\,:\, \omega_u''(1)>0\r\},\\
	\mathcal{N}_\lambda^0
	&=\bigg\{u \in \mathcal{N}_\lambda\,:\, (p+\alpha-1)\|u\|_{1,p} ^p+(q+\alpha-1)\big(\|\nabla u\|_{q,\mu}^q+\|u\|_{q,\mu}^q \big)=\lambda (\nu+\alpha-1) \|u\|_{\nu}^{\nu}\bigg\},\\
	& =\l\{u\in\mathcal{N}_\lambda\,:\, \omega_u''(1)=0\r\},\\
	\mathcal{N}_\lambda^-
	&=\bigg\{u \in \mathcal{N}_\lambda\,:\, (p+\alpha-1)\|u\|_{1,p} ^p+(q+\alpha-1)\big(\|\nabla u\|_{q,\mu}^q+\|u\|_{q,\mu}^q \big)<\lambda (\nu+\alpha-1) \|u\|_{\nu}^{\nu}\bigg\}\\
	& =\l\{u\in\mathcal{N}_\lambda\,:\, \omega_u''(1)<0\r\}.
\end{align*}

The next proposition will show that the functional $\Psi_\lambda$ restricted to $\mathcal{N}_\lambda$ is coercive.

\begin{proposition}\label{prop_coerivity}
	Let hypotheses \textnormal{(H)} be satisfied. Then $\Psi_\lambda\big|_{\mathcal{N}_\lambda}$ is coercive.
\end{proposition}

\begin{proof}
	Taking $u \in \mathcal{N}_\lambda$ with $\|u\|>1$, we obtain from the definition of $\mathcal{N}_\lambda$ that
	\begin{align}\label{prop_1}
		-\frac{\lambda}{\nu}\|u\|_{\nu}^{\nu}=-\frac{1}{\nu}  \big(\|u\|_{1,p} ^p+\|\nabla u\|_{q,\mu}^q+\|u\|_{q,\mu}^q\big)+\frac{1}{\nu}\into \xi(x)|u|^{1-\alpha}\,\diff x.
	\end{align}
	Combining \eqref{prop_1} with $\Psi_\lambda$, using Proposition \ref{proposition_modular_properties} \textnormal{(iv)} along with H\"older's inequality, see hypothesis \textnormal{(H)(iii)},  and applying the embedding \eqref{embedding} leads to
	\begin{align*}
		\Psi_\lambda(u)&=\l[\frac{1}{p}-\frac{1}{\nu} \r]\|u\|_{1,p} ^p+\l[\frac{1}{q}-\frac{1}{\nu} \r]\big(\|\nabla u\|_{q,\mu}^q+\|u\|_{q,\mu}^q \big)
		+\l[\frac{1}{\nu}-\frac{1}{1-\alpha} \r]\into \xi(x)|u|^{1-\alpha}\,\diff x\\
		& \geq \l[\frac{1}{q}-\frac{1}{\nu} \r]\varrho(u) +\l[\frac{1}{\nu}-\frac{1}{1-\alpha} \r]\|\xi\|_{\sigma}\|u\|_{\kappa}^{1-\alpha}\\
		& \geq c_1 \|u\|^p-c_2\|u\|^{1-\alpha}
	\end{align*}
	for some constants $c_1,c_2>0$ since $p<q<\nu$ by hypotheses \textnormal{(H)(i), (ii)}. Therefore, because of $1-\alpha<1<p$, the  assertion is proved.
\end{proof}

We define $m_\lambda^+=\inf_{\mathcal{N}_\lambda^+}\Psi_\lambda$.

\begin{proposition}\label{prop_negative_energy}
	Let hypotheses \textnormal{(H)} be satisfied and suppose that $\mathcal{N}_\lambda^+\neq \emptyset$. Then $m_\lambda^+<0$.
\end{proposition}

\begin{proof}
	Choosing $u \in \mathcal{N}_\lambda^+\neq \emptyset$ and noticing that $\mathcal{N}_\lambda^+\subseteq \mathcal{N}_\lambda$ gives
	\begin{align}\label{prop_3}
		-\frac{1}{1-\alpha}\into \xi(x)|u|^{1-\alpha}\,\diff x =-\frac{1}{1-\alpha} \big(\|u\|_{1,p}^p+\|\nabla u\|_{q,\mu}^q+\|u\|_{q,\mu}^q\big)+\frac{\lambda}{1-\alpha}\|u\|_{\nu}^{\nu}.
	\end{align}
	Since $u \in \mathcal{N}_\lambda^+$, we have
	\begin{align}\label{prop_2}
		\lambda \|u\|_{\nu}^{\nu}<\frac{p+\alpha-1}{\nu+\alpha-1}\|u\|_{1,p}^p+\frac{q+\alpha-1}{\nu+\alpha-1}\big(\|\nabla u\|_{q,\mu}^q+\|u\|_{q,\mu}^q\big).
	\end{align}
	Combining \eqref{prop_3} and \eqref{prop_2} with the definition of $\Psi_\lambda$ leads to
	\begin{align*}
		\begin{split}
			\Psi_{\lambda}(u)&=\frac{1}{p} \|u\|_{1,p}^p+\frac{1}{q}\big(\|\nabla u\|_{q,\mu}^q+\|u\|_{q,\mu}^q\big)-\frac{1}{1-\alpha}\into \xi(x)|u|^{1-\alpha}\,\diff x-\frac{\lambda}{\nu}\|u\|_{\nu}^{\nu}\\
			& =\l[\frac{1}{p} -\frac{1}{1-\alpha}\r]\|u\|_{1,p}^p+\l[\frac{1}{q} -\frac{1}{1-\alpha}\r]\big(\|\nabla u\|_{q,\mu}^q+\|u\|_{q,\mu}^q\big)+\lambda \l[\frac{1}{1-\alpha}-\frac{1}{\nu}\r]\|u\|_{\nu}^{\nu}\\
			&< \l[\frac{-(p+\alpha-1)}{p(1-\alpha)}+\frac{p+\alpha-1}{\nu+\alpha-1}\cdot \frac{\nu+\alpha-1}{\nu(1-\alpha)}\r]\|u\|_{1,p}^p\\
			&\quad +\l[\frac{-(q+\alpha-1)}{q(1-\alpha)}+\frac{q+\alpha-1}{\nu+\alpha-1}\cdot \frac{\nu+\alpha-1}{\nu(1-\alpha)}\r]\big(\|\nabla u\|_{q,\mu}^q+\|u\|_{q,\mu}^q\big)\\
			& = \frac{p+\alpha-1}{1-\alpha}\l[\frac{1}{\nu}-\frac{1}{p}\r]\|u\|_{1,p}^p+\frac{q+\alpha-1}{1-\alpha}\l[\frac{1}{\nu}-\frac{1}{q}\r]\big(\|\nabla u\|_{q,\mu}^q+\|u\|_{q,\mu}^q\big)\\
			&<0,
		\end{split}
	\end{align*}
	due to $p<q<\nu$. Therefore, $\Psi_\lambda \big|_{\mathcal{N}_\lambda^+}<0$ which implies $m_\lambda^+<0$.
\end{proof}

\begin{proposition}\label{prop_emptiness}
	Let hypotheses \textnormal{(H)} be satisfied. Then there exists $\overline{\lambda}>0$ such that $\mathcal{N}^0_\lambda=\emptyset$ for all $(0,\overline{\lambda})$.
\end{proposition}

\begin{proof}
	We are going to prove the assertion via contradiction. To this end, let us assume that for every $\overline{\lambda}>0$ there exists $\lambda \in (0,\overline{\lambda})$ such that $\mathcal{N}^0_\lambda \neq \emptyset$. Thus, for any given $\lambda>0$ there exists an element $u\in \mathcal{N}_\lambda^0$ such that
	\begin{align}\label{prop_4}
		(p+\alpha-1)\|u\|_{1,p}^p+(q+\alpha-1)\big(\|\nabla u\|_{q,\mu}^q+\|u\|_{q,\mu}^q\big)=\lambda (\nu+\alpha-1) \|u\|_{\nu}^{\nu}.
	\end{align}
	Because of $u \in \mathcal{N}_\lambda$ we have
	\begin{align}\label{prop_5}
		\begin{split}
		&(\nu+\alpha-1) \|u\|_{1,p}^p+(\nu+\alpha-1) \big(\|\nabla u\|_{q,\mu}^q+\|u\|_{q,\mu}^q\big)\\
		& = (\nu+\alpha-1)\into \xi(x)|u|^{1-\alpha}\,\diff x+\lambda (\nu+\alpha-1)\|u\|_{\nu}^{\nu}.
		\end{split}
	\end{align}
	Subtracting \eqref{prop_4} from \eqref{prop_5} results in
	\begin{align}\label{prop_6}
		\begin{split}
			&(\nu-p) \|u\|_{1,p}^p+(\nu-q) \big(\|\nabla u\|_{q,\mu}^q+\|u\|_{q,\mu}^q\big)= (\nu+\alpha-1)\into \xi(x)|u|^{1-\alpha}\,\diff x.
		\end{split}
	\end{align}
	From \eqref{prop_6} we obtain by using Proposition \ref{proposition_modular_properties}\textnormal{(iii)}, \textnormal{(iv)}, hypothesis \textnormal{H(iii)} and \eqref{embedding} that
	\begin{align*}
		\min\l\{\|u\|^p,\|u\|^q\r\} \leq c_3 \|u\|^{1-\alpha}
	\end{align*}
	for some $c_3>0$ since $1-\alpha<1<p<q<\nu$. Therefore,
	\begin{align}\label{prop_7}
		\|u\| \leq c_4
	\end{align}
	for some $c_4>0$.

	On the other hand, from \eqref{prop_4}, taking Proposition \ref{proposition_modular_properties}\textnormal{(iii)}, \textnormal{(iv)} and \eqref{embedding} into account, we obtain
	\begin{align*}
		\min\l\{\|u\|^p,\|u\|^q\r\} \leq \lambda c_5 \|u\|^{\nu}
	\end{align*}
	for some $c_5>0$. This leads to
	\begin{align*}
		\|u\| \geq \l(\frac{1}{\lambda c_5}\r)^{\frac{1}{\nu-p}}
		\quad\text{or}\quad
		\|u\| \geq \l(\frac{1}{\lambda c_5}\r)^{\frac{1}{\nu-q}}.
	\end{align*}
	Due to $p<q<\nu$, we see that if $\lambda \to 0^+$, then $\|u\|\to +\infty$ contradicting \eqref{prop_7}.
\end{proof}

\begin{proposition}\label{prop_nonemptiness_and_existence}
	Let hypotheses \textnormal{(H)} be satisfied. Then there exists $\hat{\lambda}\in (0,\overline{\lambda}]$ such that $\mathcal{N}^{\pm}_\lambda\neq \emptyset$ for all $\lambda\in(0,\hat{\lambda})$. Furthermore, for any $\lambda \in (0,\hat{\lambda})$, there exists $u^*\in\mathcal{N}_\lambda^+$ such that $\Psi_\lambda(u^*)=m_\lambda^+<0$ and $u^*(x) \geq 0$ for a.\,a.\,$x\in\R^N$.
\end{proposition}

\begin{proof}
	Let $u\in \WH\setminus \{0\}$. We introduce the function $\hat{\psi}_u\colon (0,+\infty) \to \R$ given by
	\begin{align*}
		\hat{\psi}_u(t)= t^{p-\nu}\|u\|_{1,p}^p-t^{-\nu-\alpha+1}\into \xi(x)|u|^{1-\alpha} \,\diff x.
	\end{align*}
	First note, since $\nu-p<\nu+\alpha-1$, there exists $\hat{t}_0>0$ such that
	\begin{align*}
		\hat{\psi}_u\l(\hat{t}_0\r)=\max_{t>0} \hat{\psi}_u(t).
	\end{align*}
	Thus, $\hat{\psi}'_u(\hat{t}_0)=0$, or equivalently
	\begin{align*}
		(p-\nu)\hat{t}_0^{p-\nu-1}\|u\|_{1,p}^p+(\nu+\alpha-1)\hat{t}_0^{-\nu-\alpha}\into \xi(x)|u|^{1-\alpha} \,\diff x=0,
	\end{align*}
	which implies
	\begin{align*}
		\hat{t}_0=\l[\frac{(\nu+\alpha-1)\into \xi(x)|u|^{1-\alpha} \,\diff x}{(\nu-p)\|u\|_{1,p}^p}\r]^{\frac{1}{p+\alpha-1}}.
	\end{align*}
	Furthermore, we obtain
	\begin{align}\label{prop_40}
		\begin{split}
			\hat{\psi}_u\l(\hat{t}_0\r)
			&=\frac{\Big[(\nu-p) \|u\|_{1,p}^p \Big]^{\frac{\nu-p}{p+\alpha-1}}}{\Big[(\nu+\alpha-1)\into \xi(x)|u|^{1-\alpha}\,\diff x \Big]^{\frac{\nu-p}{p+\alpha-1}}}\|u\|_{1,p}^p\\
			&\quad -\frac{\Big[(\nu-p) \|u\|_{1,p}^p \Big]^{\frac{\nu+\alpha-1}{p+\alpha-1}}}{\Big[(\nu+\alpha-1)\into \xi(x)|u|^{1-\alpha}\,\diff x \Big]^{\frac{\nu+\alpha-1}{p+\alpha-1}}}\into \xi(x)|u|^{1-\alpha}\,\diff x\\
			&=\frac{(\nu-p)^{\frac{\nu-p}{p+\alpha-1}} {\l(\|u\|_{1,p}^p\r) }^{\frac{\nu+\alpha-1}{p+\alpha-1}}}{(\nu+\alpha-1)^{\frac{\nu-p}{p+\alpha-1}}\Big[\into \xi(x)|u|^{1-\alpha}\,\diff x \Big]^{\frac{\nu-p}{p+\alpha-1}}}\\
			&\quad -\frac{(\nu-p)^{\frac{\nu+\alpha-1}{p+\alpha-1}} {\l(\|u\|_{1,p}^p\r) }^{\frac{\nu+\alpha-1}{p+\alpha-1}}}{(\nu+\alpha-1)^{\frac{\nu+\alpha-1}{p+\alpha-1}}\Big[\into \xi(x)|u|^{1-\alpha}\,\diff x \Big]^{\frac{\nu-p}{p+\alpha-1}}}\\
			&=\frac{p+\alpha-1}{\nu-p} \l[\frac{\nu-p}{\nu+\alpha-1}\r]^{\frac{\nu+\alpha-1}{p+\alpha-1}}\frac{{\l(\|u\|_{1,p}^p\r)}^{\frac{\nu+\alpha-1}{p+\alpha-1}}}{\Big[\into \xi(x)|u|^{1-\alpha}\,\diff x \Big]^{\frac{\nu-p}{p+\alpha-1}}}.
		\end{split}	
	\end{align}
	Moreover, from hypothesis \textnormal{H(iii)}, H\"older's inequality and the continuous embedding $\Wp{p}$ $\hookrightarrow$ $\Lp{\kappa}$, we get
	\begin{align}\label{prop_42}
		\into \xi(x)|u|^{1-\alpha}\,\diff x \leq c_6 \|u\|_{1,p}^{1-\alpha}
	\end{align}
	for some $c_6>0$. Combining \eqref{prop_40} and \eqref{prop_42} along with the continuous embedding $\Wp{p}\hookrightarrow \Lp{\nu}$ leads to
	\begin{align*}
		&\hat{\psi}_u\l(\hat{t}_0\r)-\lambda \|u\|_{\nu}^{\nu}\\
		&=\frac{p+\alpha-1}{\nu-p} \l[\frac{\nu-p}{\nu+\alpha-1}\r]^{\frac{\nu+\alpha-1}{p+\alpha-1}}\frac{{\l(\|u\|_{1,p}^p\r) }^{\frac{\nu+\alpha-1}{p+\alpha-1}}}{\Big[\into \xi(x)|u|^{1-\alpha}\,\diff x \Big]^{\frac{\nu-p}{p+\alpha-1}}}-\lambda \|u\|_{\nu}^{\nu}\\
		&\geq \frac{p+\alpha-1}{\nu-p} \l[\frac{\nu-p}{\nu+\alpha-1}\r]^{\frac{\nu+\alpha-1}{p+\alpha-1}}\frac{\l(\| u\|_{1,p}^p\r)^{\frac{\nu+\alpha-1}{p+\alpha-1}}}{\l(c_6 \|u\|_{1,p}^{1-\alpha}\r)^{\frac{\nu-p}{p+\alpha-1}}}-\lambda c_7\|u\|_{1,p}^{\nu}\\
		&= \Big [c_8-\lambda c_7\Big] \|u\|_{1,p}^{\nu}
	\end{align*}
	for some $c_7, c_8>0$. Hence, we can find $\hat{\lambda} \in (0,\overline{\lambda}]$ independent of $u$ such that
	\begin{align}\label{prop_43}
		\hat{\psi}_u\l(\hat{t}_0\r)-\lambda \|u\|_{\nu}^{\nu}>0 \quad\text{for all }\lambda \in \big(0,\hat{\lambda}\big).
	\end{align}

	Let us now introduce the function $\psi_u\colon (0,+\infty) \to \R$ given by
	\begin{align*}
		\psi_u(t)= t^{p-\nu}\|u\|_{1,p}^p+t^{q-\nu}\big(\|\nabla u\|_{q,\mu}^q+\|u\|_{q,\mu}^q\big)-t^{-\nu-\alpha+1}\into \xi(x)|u|^{1-\alpha} \,\diff x.
	\end{align*}
	Due to $\nu-q<\nu-p<\nu+\alpha-1$ there exists $t_0>0$ such that
	\begin{align*}
		\psi_u(t_0)=\max_{t>0} \psi_u(t).
	\end{align*}
	Since $\psi_u \geq \hat{\psi}_u$ and because of \eqref{prop_43} we are able to find $\hat{\lambda} \in (0,\overline{\lambda}]$ independent of $u$ such that
	\begin{align*}
		\psi_u\l(t_0\r)-\lambda \|u\|_{\nu}^{\nu}>0 \quad\text{for all }\lambda \in \big(0,\hat{\lambda}\big).
	\end{align*}
	Hence, there exist unique $t_u^1<t_0<t_u^2$ satisfying
	\begin{equation}\label{prop_8}
		\psi_u(t_u^1)=\lambda \|u\|_{\nu}^{\nu}=\psi_u(t_u^2)
		\quad \text{and}\quad
		\psi'_u(t_u^2)<0<\psi'_u(t_u^1),
	\end{equation}
	where
	\begin{align}\label{prop_7b}
		\begin{split}
			\psi'_u(t)
			&=(p-\nu)t^{p-\nu-1}\|u\|_{1,p}^p+(q-\nu)t^{q-\nu-1}\big(\|\nabla u\|_{q,\mu}^q+\|u\|_{q,\mu}^q\big)\\
			&\quad -(-\nu-\alpha+1)t^{-\nu-\alpha}\into \xi(x)|u|^{1-\alpha}\,\diff x.
		\end{split}
	\end{align}
	
	Recall that the fibering function $\omega_u\colon[0,+\infty)\to \R$ is given by
	\begin{align*}
		\omega_u(t)=\Psi_\lambda (tu)\quad\text{for all }t\geq 0.
	\end{align*}
	We have
	\begin{align*}
		\omega'_u\l(t_u^1\r)
		&=\l( t_u^1\r)^{p-1}\|u\|_{1,p}^p+\l( t_u^1\r)^{q-1}\big(\|\nabla u\|_{q,\mu}^q+\|u\|_{q,\mu}^q\big)\\
		&  \quad -\l( t_u^1\r)^{-\alpha}\into \xi(x)|u|^{1-\alpha}\,\diff x-\lambda \l( t_u^1\r)^{\nu-1}\|u\|_{\nu}^{\nu}
	\end{align*}
	and
	\begin{align}\label{prop_9}
		\begin{split}
		\omega''_u\l(t_u^1\r)
		&=(p-1)\l( t_u^1\r)^{p-2}\|u\|_{1,p}^p+(q-1)\l( t_u^1\r)^{q-2}\big(\|\nabla u\|_{q,\mu}^q+\|u\|_{q,\mu}^q\big)\\
		&\quad +\alpha \l( t_u^1\r)^{-\alpha-1}\into \xi(x)|u|^{1-\alpha}\,\diff x
		-\lambda(\nu-1) \l( t_u^1\r)^{\nu-2}\|u\|_{\nu}^{\nu}.
		\end{split}
	\end{align}

	Taking \eqref{prop_8} into account we have
	\begin{align*}
		\l( t_u^1\r)^{p-\nu}\|u\|_{1,p}^p +\l( t_u^1\r)^{q-\nu}\big(\|\nabla u\|_{q,\mu}^q+\|u\|_{q,\mu}^q\big)-\l( t_u^1\r)^{-\nu-\alpha+1}\into \xi(x)|u|^{1-\alpha}\,\diff x=\lambda \|u\|_{\nu}^{\nu}.
	\end{align*}
	We multiply the equation above with $\alpha \l(t_u^1\r)^{\nu-2}$ and $-(\nu-1)\l(t_u^1\r)^{\nu-2}$, respectively. This gives
	\begin{align}\label{prop_10}
		\begin{split}
			&\alpha \l(t_u^1\r)^{p-2}\|u\|_{1,p}^p +\alpha \l(t_u^1\r)^{q-2}\big(\|\nabla u\|_{q,\mu}^q+\|u\|_{q,\mu}^q\big)-\alpha\lambda \l(t_u^1\r)^{\nu-2} \|u\|_{\nu}^{\nu}\\
			&= \alpha\l(t_u^1\r)^{-\alpha-1}\into \xi(x)|u|^{1-\alpha}\,\diff x
		\end{split}	
	\end{align}
	and
	\begin{align}\label{prop_11}
		\begin{split}
			&-(\nu-1)\l( t_u^1\r)^{p-2}\|u\|_{1,p}^p -(\nu-1)\l( t_u^1\r)^{q-2}\big(\|\nabla u\|_{q,\mu}^q+\|u\|_{q,\mu}^q\big)\\
			&\quad +(\nu-1)\l( t_u^1\r)^{-\alpha-1}\into \xi(x)|u|^{1-\alpha}\,\diff x\\
			&=-\lambda(\nu-1) \l(t_u^1\r)^{\nu-2} \|u\|_{\nu}^{\nu}.
		\end{split}	
	\end{align}
	
	Now we combine  \eqref{prop_9}  and \eqref{prop_10} in order to obtain
	\begin{align}\label{prop_12}
		\begin{split}
			\omega''_u\l(t_u^1\r)
			&=(p+\alpha -1)\l( t_u^1\r)^{p-2}\|u\|_{1,p}^p+(q+\alpha-1)\l( t_u^1\r)^{q-2}\big(\|\nabla u\|_{q,\mu}^q+\|u\|_{q,\mu}^q\big)\\
			&\quad -\lambda(\nu+\alpha -1) \l( t_u^1\r)^{\nu-2}\|u\|_{\nu}^{\nu}\\
			&= \l( t_u^1\r)^{-2}\Big[(p+\alpha -1)\l( t_u^1\r)^{p}\|u\|_{1,p}^p+(q+\alpha-1)\l( t_u^1\r)^{q}\big(\|\nabla u\|_{q,\mu}^q+\|u\|_{q,\mu}^q\big)\\
			&\qquad\qquad \quad-\lambda(\nu+\alpha -1) \l( t_u^1\r)^{\nu}\|u\|_{\nu}^{\nu}\Big].
		\end{split}
	\end{align}

	Applying \eqref{prop_7b}, \eqref{prop_9} and \eqref{prop_11} results in
	\begin{equation}\label{prop_13}
		\begin{split}
			\omega''_u\l(t_u^1\r)&=(p-\nu)\l( t_u^1\r)^{p-2}\|u\|_{1,p}^p+(q-\nu)\l( t_u^1\r)^{q-2}\big(\|\nabla u\|_{q,\mu}^q+\|u\|_{q,\mu}^q\big)\\
			&\quad +(\nu+\alpha-1) \l( t_u^1\r)^{-\alpha-1}\into \xi(x)|u|^{1-\alpha}\,\diff x\\
			&=\l(t^1_u\r)^{\nu-1} \psi'_u\l( t_u^1\r)>0.
		\end{split}
	\end{equation}
	Combining \eqref{prop_12} and \eqref{prop_13} we see that
	\begin{align*}
		\begin{split}
			(p+\alpha -1)\l( t_u^1\r)^{p}\|u\|_{1,p}^p+(q+\alpha-1)\l( t_u^1\r)^{q}\big(\|\nabla u\|_{q,\mu}^q+\|u\|_{q,\mu}^q\big)-\lambda(\nu+\alpha -1) \l( t_u^1\r)^{\nu}\|u\|_{\nu}^{\nu}>0.
		\end{split}
	\end{align*}
	Therefore,
	\begin{equation*}
		t_u^1 u\in \mathcal{N}_\lambda^+ \quad \text{for all } \lambda\in \l(0,\hat{\lambda}\r]
	\end{equation*}
	and so, $\mathcal{N}_\lambda^+\neq \emptyset$. A similar treatment can be used for the point $t_u^2$ in order to show that $\mathcal{N}_\lambda^-\neq \emptyset$.
	
	We are now going to prove the second assertion of the proposition. For this purpose, let $\{u_n\}_{n\in\N}\subset \mathcal{N}_\lambda^+$ be a minimizing sequence, that is,
	\begin{align}\label{prop_14}
		\Psi_\lambda(u_n) \searrow m^+_\lambda <0 \quad\text{as }n\to\infty.
	\end{align}
	Since $\mathcal{N}_\lambda^+ \subset \mathcal{N}_\lambda$ we deduce from Proposition \ref{prop_coerivity} that $\{u_n\}_{n\in\N}\subset \WH$ is bounded. So, we can assume that
	\begin{align}\label{prop_15}
		u_n\weak u^* \quad\text{in }\WH,
		\quad u_n\to u^* \quad \text{in }L^\nu_{\loc}(\R^N)
		\quad\text{and}\quad
		u_n\to u^* \quad\text{a.\,e.\,in }\R^N.
	\end{align}
	From \eqref{prop_14}, \eqref{prop_15} and Proposition \ref{prop_negative_energy} along with the first part of the proof we see that
	\begin{align*}
		\Psi_\lambda(u^*) \leq \liminf_{n\to+\infty} \Psi_\lambda(u_n)<0=\Psi_\lambda(0).
	\end{align*}
	Thus, $u^*\neq 0$.
	
	Let us now prove that
	\begin{align}\label{prop_1045}
		\liminf_{n\to+\infty} \varrho(u_n)=\varrho\l(u^*\r).
	\end{align}

	We argue by contradiction and suppose that
	\begin{align*}
		\liminf_{n\to+\infty} \varrho(u_n)>\varrho(u^*).
	\end{align*}
	Applying the representation in \eqref{prop_8}, we deduce that
	\begin{align*}
		\begin{split}
			&\liminf_{n\to+\infty} \omega'_{u_n}(t_{u^*}^1)\\			&=\liminf_{n\to+\infty}\l[\l(t_{u^*}^1\r)^{p-1}\|u_n\|_{1,p}^p+\l(t_{u^*}^1\r)^{q-1}\big(\|\nabla u_n\|_{q,\mu}^q+\|u_n\|_{q,\mu}^q\big) -\l(t_{u^*}^1\r)^{-\alpha}\into \xi(x)|u_n|^{1-\alpha}\,\diff x\r.\\
			&\qquad\qquad\quad \l.-\lambda \l(t_{u^*}^1\r)^{\nu-1}\|u_n\|_{\nu}^{\nu}\r]\\
			&>\l(t_{u^*}^1\r)^{p-1}\l\|u^*\r\|_{1,p}^p+\l(t_{u^*}^1\r)^{q-1}\big(\l\|\nabla u^*\r\|_{q,\mu}^q+\l\|u^*\r\|_{q,\mu}^q\big) -\l(t_{u^*}^1\r)^{-\alpha}\into \xi(x)|u^*|^{1-\alpha}\,\diff x\\
			&\quad -\lambda \l(t_{u^*}^1\r)^{\nu-1}\|u^*\|_{\nu}^{\nu}\\
			&=\omega'_{u^*}(t_{u^*}^1)=t_{u^*}^{\nu-1}\l[\psi_{u^*}\l(t_{u^*}^1\r)-\lambda\|u^*\|_{\nu}^{\nu}\r]=0.
		\end{split}
	\end{align*}
	From this we conclude that there exists $n_0\in \N$ such that $\omega'_{u_n}\l(t_{u^*}^1\r)>0$ for all $n>n_0$. Since $u_n\in \mathcal{N}^+_{\lambda}\subset \mathcal{N}_{\lambda}$ and $\omega'_{u_n}(t)=t^{\nu-1} \l[\psi_{u_n}(t)-\lambda\|u_n\|_{\nu}^{\nu}\r]$ we derive that $\omega'_{u_n}(t)<0$ for all $t\in(0,1)$ and $\omega'_{u_n}(1)=0$. Hence, $t_{u^*}^1>1$.
	
	We know that $\omega_{u^*}$ is decreasing on $[0,t_{u^*}^1]$. Therefore,
	\begin{align*}
		\Psi_{\lambda} \l(t_{u^*}^1 u^*\r) \leq \Psi_{\lambda}\l(u^*\r)<m^+_{\lambda}.
	\end{align*}
	Since $t_{u^*}^1 u^*\in \mathcal{N}^+_{\lambda}$ we then obtain
	\begin{align*}
		m^+_{\lambda}\leq \Psi_{\lambda}\l(t_{u^*}^1 u^*\r)<m^+_{\lambda},
	\end{align*}
	which is a contradiction. This shows \eqref{prop_1045}.
	
	From \eqref{prop_1045} we conclude that that there exists a subsequence (still denoted by $u_n$) such that $\varrho(u_n)\to\varrho(u^*)$. Using Proposition \ref{proposition_modular_properties}\textnormal{(v)} we see that $u_n \to u$ in $\WH$ and so $\Psi_{\lambda}(u_n)\to \Psi_{\lambda}(u^*)$. Therefore, $\Psi_{\lambda}(u^*)=m^+_{\lambda}$. As $u_n\in \mathcal{N}^+_{\lambda}$ for all $n\in \N$, we have
	\begin{align*}
		(p+\alpha-1)\|u_n\|_{1,p}^p+(q+\alpha-1)\big(\|\nabla u_n\|_{q,\mu}^q+\|u_n\|_{q,\mu}^q\big)-\lambda (\nu+\alpha-1) \|u_n\|_{\nu}^{\nu}>0,
	\end{align*}
	which by letting $n\to+\infty$ results in
	\begin{equation}\label{prop_16}
		(p+\alpha-1)\l\|u^*\r\|_{1,p}^p+(q+\alpha-1)\big(\l\|\nabla u^*\r\|_{q,\mu}^q+\l\|u^*\r\|_{q,\mu}^q\big)-\lambda (\nu+\alpha-1) \l\|u^*\r\|_{\nu}^{\nu}\geq 0.
	\end{equation}
	Since $\lambda\in (0,\hat{\lambda})$ and $\hat{\lambda}\leq \overline{\lambda}$ we are able to apply Proposition \ref{prop_emptiness} in order to know that the left-hand side in \eqref{prop_16} is strictly positive. Consequently, we deduce that $u^*\in \mathcal{N}^+_{\lambda}$. As we can always use $|u^*|$ instead of $u^*$, we may assume that $u^*(x)\geq 0$ for a.\,a.\,$x\in\R^N$ with $u^*\neq 0$.  This finishes the proof.
\end{proof}

\begin{proposition}\label{prop_energy_estimate}
	Let hypotheses \textnormal{(H)} be satisfied, let $v\in\WH$ and let $\lambda \in(0, \hat{\lambda}]$. Then there exists $\zeta>0$  such that $\Psi_{\lambda}(u^*)\leq \Psi_{\lambda}(u^*+tv)$ for all $t\in [0,\zeta]$.
\end{proposition}

\begin{proof}
	Let $u\in \mathcal{N}_\lambda^+$ and let $\Xi\colon\WH\times (0,\infty)\to\R$ be defined by
	\begin{align*}
		\Xi(y,t)
		&=t^{p+\alpha-1}\|u+y\|_{1,p}^p+t^{q+\alpha-1}\Big(\|\nabla (u+y)\|_{q,\mu}^q+\|u+y\|_{q,\mu}^q\Big)-\into \xi(x)|u+y|^{1-\alpha}\,\diff x\\
		&\quad -\lambda t^{\nu+\alpha-1} \|u+y\|_\nu^\nu\quad \text{for all } y\in \WH.
	\end{align*}
	First note that $\Xi(0,1)=0$ because $u\in \mathcal{N}^+_\lambda \subset \mathcal{N}_\lambda$. Furthermore, due to $u \in \mathcal{N}^+_\lambda$, we have
	\begin{align*}
		\Xi'_t(0,1)=(p+\alpha-1)\| u\|_{1,p}^p+(q+\alpha-1)\Big(\|\nabla u\|_{q,\mu}^q+\|u\|_{q,\mu}^q\Big) -\lambda (\nu+\alpha-1)\|u\|_\nu^\nu>0.
	\end{align*}
	Thus, we may apply the implicit function theorem, see Berger \cite[p.\,115]{Berger-1977}, in order to find $\eps>0$ and a continuous function $\ph\colon B_\eps(0)\to (0,\infty)$ such that
	\begin{align*}
		\ph(0)=1\quad\text{and}\quad \ph(y)(u+y) \in \mathcal{N}_\lambda \quad\text{for all } y\in B_\eps(0),
	\end{align*}
	where
	\begin{align*}
		B_\eps(0)=\Big\{u\in\WH\,:\, \|u\|<\eps\Big\}.
	\end{align*}
	If we choose $\eps>0$ sufficiently small, we can have that
	\begin{align}\label{prop_17b}
		\ph(0)=1\quad\text{and}\quad \ph(y)(u+y) \in \mathcal{N}^{+}_\lambda \quad\text{for all }y\in B_\eps(0).
	\end{align}
	
	Now we introduce the function $\eta_v\colon [0,+\infty)\to \R$ given by
	\begin{align}\label{prop_17}
		\begin{split}
			\eta_v(t)
			&=(p-1)\l \|u^*+t v\r\|_{1,p}^p+(q-1)\Big(\|\nabla u^*+t\nabla v\|_{q,\mu}^q+\|u^*+tv\|_{q,\mu}^q\Big)\\
			& \quad+\alpha \into \xi(x)\l|u^*+tv \r|^{1-\alpha}\,\diff x-\lambda (\nu-1)\l\| u^*+tv \r\|_\nu^\nu.
		\end{split}
	\end{align}
	Since $u^*$ belongs to both $\mathcal{N}_\lambda^+$ and $\mathcal{N}_\lambda$ we obtain
	\begin{align}\label{prop_18}
		\alpha \into \xi(x)\l| u^*\r|^{1-\alpha} \,\diff x=\alpha \l\| u^*\r\|_{1,p}^p+\alpha \Big(\l\|\nabla u^*\r\|_{q,\mu}^q+\l\|u^*\r\|_{q,\mu}^q\Big)-\lambda \alpha \l\| u^*\r\|_\nu^\nu
	\end{align}
	and
	\begin{align}\label{prop_19}
		(p+\alpha-1)\l\| u^*\r\|_{1,p}^p+(q+\alpha-1) \Big(\l\|\nabla u^*\r\|_{q,\mu}^q+\l\|u^*\r\|_{q,\mu}^q\Big)-\lambda(\nu+\alpha-1) \l\| u^*\r\|_\nu^\nu>0.
	\end{align}
	From \eqref{prop_17}, \eqref{prop_18} and \eqref{prop_19} we conclude that $\eta_v(0)>0$ and due to the continuity of $\eta_v\colon [0,+\infty)\to \R$ there exists $\zeta_0>0$ such that
	\begin{align*}
		\eta_v(t)>0 \quad\text{for all }t \in [0,\zeta_0].
	\end{align*}
	From the first part of the proof, see \eqref{prop_17b}, we know that for every $t \in [0,\zeta_0]$ there exists $\ph(t)>0$ such that
	\begin{align}\label{prop_20}
		\ph(t)\l(u^*+tv\r)\in \mathcal{N}_\lambda^+
		\quad\text{and}\quad
		\ph(t) \to 1 \quad\text{as } t\to 0^+.
	\end{align}
	Furthermore, Proposition \ref{prop_nonemptiness_and_existence} implies that
	\begin{align*}
		m_\lambda^+=\Psi_\lambda \l(u^*\r) \leq \Psi_\lambda \l(\ph(t)\l(u^*+tv\r)\r)\quad\text{for all } t \in [0,\zeta_0].
	\end{align*}
	Using this fact and  \eqref{prop_20} there exists $\zeta\in (0,\zeta_0]$ sufficiently small such that
	\begin{align*}
		m_\lambda^+=\Psi_\lambda \l(u^*\r) \leq \Psi_\lambda \l(u^*+tv\r)\quad\text{for all } t \in [0,\zeta].
	\end{align*}
	This follows from the fact that $\omega_{u^*}''(1)>0$ and its continuity in $t$ which gives $\omega''_{u^* +tv}(1)>0$ for $t \in [0,\zeta]$ with $\zeta\in (0,\zeta_0]$. The proof is finished.
\end{proof}

Now we are in the position to show that $u^*$ is indeed a nontrivial weak solution of problem \eqref{problem} with negative energy.

\begin{proposition}\label{prop_first_weak_solution}
	Let hypotheses \textnormal{(H)} be satisfied and let $\lambda \in(0, \hat{\lambda}]$. Then $u^*$ is a weak solution of problem \eqref{problem} such that $\Psi_\lambda(u^*)<0$.
\end{proposition}

\begin{proof}
	From Proposition \ref{prop_nonemptiness_and_existence} we already know that $u^*\geq 0$ for a.\,a.\,$x\in \R^N$ and $\Psi_\lambda(u^*)<0$. We claim that $u^*>0$ for a.\,a.\,$x\in \R^N$. Suppose this is not the case and assume that there exists a set $C$ of positive measure such that $u^*= 0$ in $C$. Let $v\in \WH$, $v > 0$, and let $t\in (0,\zeta)$ (see Proposition \ref{prop_energy_estimate}) small enough such that $(u^*+tv)^{1-\alpha}>(u^* )^{1-\alpha}$ a.\,e.\,in $\R^N\setminus C$. Therefore, from Proposition \ref{prop_energy_estimate} we get
	\begin{align*}
		0
		&\leq \frac{\Psi_\lambda(u^*+tv)-\Psi_\lambda(u^*)}{t} \\
			&=\frac{1}{p} \frac{\|u^*+tv\|_{1,p}^p-\|u^*\|_{1,p}^p}{t} +\frac{1}{q} \frac{\|\nabla (u^*+tv)\|_{q,\mu}^q-\|\nabla u^*\|_{q,\mu}^q}{t}
			+\frac{1}{q} \frac{\|u^*+tv\|_{q,\mu}^q-\|u^*\|_{q,\mu}^q}{t}\\
			& \quad-\frac{1}{(1-\alpha)t^{\alpha}} \int_C \xi(x)v^{1-\alpha} \,\diff x- \frac{1}{1-\alpha}\int_{\R^N\setminus C} \xi(x)\frac{(u^*+tv)^{1-\alpha}-(u^*)^{1-\alpha}}{t}\,\diff x\\
			& \quad -\frac{\lambda}{\nu}
			 \frac{\|u^*+tv\|_\nu^\nu-\|u^*\|_\nu^\nu}{t}\\
			&<\frac{1}{p} \frac{\|u^*+tv\|_{1,p}^p-\|u^*\|_{1,p}^p}{t} +\frac{1}{q} \frac{\|\nabla (u^*+tv)\|_{q,\mu}^q-\|\nabla u^*\|_{q,\mu}^q}{t}
			+\frac{1}{q} \frac{\|u^*+tv\|_{q,\mu}^q-\|u^*\|_{q,\mu}^q}{t}\\
			& \quad-\frac{1}{(1-\alpha)t^{\alpha}} \int_C \xi(x)v^{1-\alpha} \,\diff x-\frac{\lambda}{\nu}
			 \frac{\|u^*+tv\|_\nu^\nu-\|u^*\|_\nu^\nu}{t}.
	\end{align*}
	Hence
	\begin{align*}
		0
		&\leq \frac{\Psi_\lambda(u^*+tv)-\Psi_\lambda(u^*)}{t} \to -\infty \quad \text{as }t\to 0^+,
	\end{align*}
	which is a contradiction. Therefore, $u^*>0$ a.\,e.\,in $\R^N$.
	
	Next, we will show that
	\begin{equation}\label{def1}
		\xi(\cdot)(u^*)^{-\alpha} v\in L^1(\R^N) \quad \text{for all } v\in\WH
	\end{equation}
	and
	\begin{align}\label{def2}
		\begin{split}
			& \into \Big(|\nabla u^*|^{p-2} \nabla u^*+ \mu(x) |\nabla u^*|^{q-2} \nabla u^*\Big) \cdot \nabla v \,\diff x+\into \Big((u^*)^{p-1}+\mu(x)(u^*)^{q-1}\Big)v \,\diff x \\
			&\ge \into \xi(x)(u^*)^{-\alpha} v \,\diff x+\lambda \into (u^* )^{\nu-1}v\,\diff x\quad \text{for all } v\in \WH \text{ with }v \geq 0.
		\end{split}
	\end{align}
	Let $v\in \WH$, $v\geq 0$ and choose a decreasing sequence $\{t_n\}_{n \in\N} \subseteq (0,1]$ such that $\displaystyle \lim_{n\to \infty} t_n=0$.  For $n \in\N$, the functions
	\begin{align*}
		f_n(x)=\xi(x)\frac{(u^*(x)+t_nv(x))^{1-\alpha}-u^*(x)^{1-\alpha}}{t_n}
	\end{align*}
	are measurable, nonnegative  and it holds
	\begin{align*}
		\lim_{n\to \infty} f_n(x)=(1-\alpha) \xi(x)u^*(x)^{-\alpha}v(x)\quad \text{for a.\,a.\,} x\in\R^N.
	\end{align*}
	Applying Fatou's lemma  yields
	\begin{equation}\label{fatou}
		\into \xi(x) \ykh{u^*}^{-\alpha}v\,\diff x\leq \frac{1}{1-\alpha}\liminf_{n\to\infty}\into f_n\,\diff x.
	\end{equation}
	Using again Proposition \ref{prop_energy_estimate} one has for $n\in\N$ large enough
	\begin{align*}
		0
		&\leq \frac{\Psi_\lambda(u^*+t_nv)-\Psi_\lambda(u^*)}{t_n} \\
		&=\frac{1}{p} \frac{\|u^*+t_nv\|_{1,p}^p-\|u^*\|_{1,p}^p}{t_n} +\frac{1}{q} \frac{\|\nabla (u^*+t_nv)\|_{q,\mu}^q-\|\nabla u^*\|_{q,\mu}^q}{t_n}
		+\frac{1}{q} \frac{\|u^*+t_nv\|_{q,\mu}^q-\|u^*\|_{q,\mu}^q}{t_n}\\
		& \quad- \frac{1}{1-\alpha}\int_{\R^N} f_n\,\diff x-\frac{\lambda}{\nu}
		 \frac{\|u^*+t_nv\|_\nu^\nu-\|u^*\|_\nu^\nu}{t_n}.
	\end{align*}
	Passing to the limit as $n\to \infty$ and applying \eqref{fatou} we obtain first \eqref{def1} and it also follows \eqref{def2}. We point out that it is sufficient to show \eqref{def1} for nonnegative $v\in \WH$.

	Now, we can conclude that $u^*$ is a weak solution of \eqref{problem}, see, for example, Farkas-Winkert \cite[Proof of Theorem 1.1]{Farkas-Winkert-2021} for the full calculation. We skip these long calculations as it is quite standard.
\end{proof}

Now we are interested in a second nontrivial solution of \eqref{problem} by applying the manifold $\mathcal{N}^-_\lambda$.

\begin{proposition}\label{prop_minimizer_negative_manifold}
	Let hypotheses \textnormal{(H)} be satisfied. Then there exists $\lambda^* \in(0, \hat{\lambda}]$ such that $\Psi_\lambda\big|_{\mathcal{N}^-_\lambda} > 0$ for all $\lambda \in(0, \lambda^*]$.
\end{proposition}

\begin{proof}
	First note that $\mathcal{N}_\lambda^-\neq \emptyset$ by Proposition \ref{prop_nonemptiness_and_existence} and so we can pick $u \in \mathcal{N}_\lambda^-$. Applying the continuous embedding $\Wp{p}\to \Lp{\nu}$ and the definition of the manifold $\mathcal{N}_\lambda^-$ it is easy to see that
	\begin{align*}
		\begin{split}
			\lambda (\nu+\alpha-1) c_9^\nu\|u\|_{1,p}^{\nu}&\ge\lambda (\nu+\alpha-1) \|u\|_{\nu}^{\nu}\\
			&>(p+\alpha-1)\|u\|_{1,p}^p+(q+\alpha-1)\big(\|\nabla u\|_{q,\mu}^q+\|u\|_{q,\mu}^q\big)\\
			&\geq (p+\alpha-1) \|u\|_{1,p}^p
		\end{split}
	\end{align*}
	for some $c_9>0$ which gives
	\begin{equation}\label{prop_23}
		\|u\|_{1,p}\geq
		\displaystyle \l[\frac{p+\alpha-1}{\lambda c_9^\nu(\nu+\alpha-1)}\r]^{\frac{1}{\nu-p}}.
	\end{equation}
	
	Let us now suppose that the assertion of the proposition is not true. Then there exists $u \in \mathcal{N}_\lambda^-$ such that $\Psi_\lambda(u)\leq 0$, which means,
	\begin{equation}\label{prop_24}
		\frac{1}{p}\|u\|_{1,p}^p +\frac{1}{q} \big(\|\nabla u\|_{q,\mu}^q+\|u\|_{q,\mu}^q\big)-\frac{1}{1-\alpha} \into \xi(x)|u|^{1-\alpha}\,\diff x-\frac{\lambda}{\nu}\|u\|_{\nu}^{\nu}\leq 0.
	\end{equation}
	On the other hand, as $\mathcal{N}_\lambda^-\subseteq \mathcal{N}_\lambda$, we obtain
	\begin{equation}\label{prop_25}
		\frac{1}{q}\big(\|\nabla u\|_{q,\mu}^q+\|u\|_{q,\mu}^q\big) =\frac{1}{q}\into \xi(x)|u|^{1-\alpha}\,\diff x+\frac{\lambda}{q} \|u\|_{\nu}^{\nu} -\frac{1}{q}\|u\|_{1,p}^p.
	\end{equation}
	Now we can use \eqref{prop_25} in \eqref{prop_24} which results in
	\begin{align*}
		\l(\frac{1}{p}-\frac{1}{q}\r)\|u\|_{1,p}^p+
		\l(\frac{1}{q}-\frac{1}{1-\alpha}\r)\into \xi(x)|u|^{1-\alpha}\,\diff x+\lambda \l(\frac{1}{q}-\frac{1}{\nu}\r)\|u\|_{\nu}^{\nu}\leq 0.
	\end{align*}
	Since $p<q<\nu$ and by applying hypothesis \textnormal{(H)(ii)} we derive from the inequality above that
	\begin{align*}
		 \frac{q-p}{pq}\|u\|_{1,p}^p\leq \frac{q+\alpha -1}{q(1-\alpha)}\into \xi(x) |u|^{1-\alpha}\,\diff x \leq \frac{q+\alpha -1}{q(1-\alpha)}c_{10}\|u\|_{1,p}^{1-\alpha}
	\end{align*}
	for some $c_{10}>0$. Hence,
	\begin{equation}\label{prop_26}
		\|u\|_{1,p} \leq c_{11}
	\end{equation}
	for some $c_{11}>0$. Using \eqref{prop_26}  in \eqref{prop_23} then yields
	\begin{align*}
		0<\frac{c_{12}}{c_{11}}\leq \lambda^{\frac{1}{\nu-p}}  \quad\text{with }
		\quad  c_{12}=\l[\frac{p+\alpha-1}{c_9^\nu(\nu+\alpha-1)}\r]^{\frac{1}{\nu-p}}>0.
	\end{align*}
	Letting $\lambda\to 0^+$ gives a contradiction as $1<p<\nu$. Therefore, there exists $\lambda^* \in(0, \hat{\lambda}]$ such that $\Psi_\lambda\big|_{\mathcal{N}^-_\lambda} > 0$ for all $\lambda \in(0, \lambda^*]$.
\end{proof}

Now we minimize $\Psi_\lambda$ on the manifold $\mathcal{N}_\lambda^-$. To this end, let $m_\lambda^-=\inf_{\mathcal{N}_\lambda^-}\Psi_\lambda$.

\begin{proposition}\label{prop_minimize_negative}
	Let hypotheses \textnormal{(H)} be satisfied and let $\lambda \in(0, \lambda^*]$. Then there exists $v^*\in\mathcal{N}_\lambda^-$ with $v^*\geq 0$ such that $m_\lambda^-=\Psi_\lambda\l(v^*\r)>0$.
\end{proposition}

\begin{proof}
	For $v\in \mathcal{N}_\lambda^-\neq \emptyset$, using the embedding $\Wp{p} \hookrightarrow \Lp{\nu}$, we obtain
	\begin{align*}
		\begin{split}
			&\lambda (\nu+\alpha-1) \|v\|_{\nu}^{\nu}
			\geq(p+\alpha-1)\|v\|_{1,p}^p \geq (p+\alpha-1)\frac{1}{c_9^p}\|v\|_{\nu}^p
		\end{split}
	\end{align*}
	for $c_9>0$, see the proof of Proposition \ref{prop_minimizer_negative_manifold}. Therefore,
	\begin{equation}\label{prop_235}
		\|v\|_{\nu}\geq
		\displaystyle \l[\frac{p+\alpha-1}{\lambda c_9^{p}(\nu+\alpha-1)}\r]^{\frac{1}{\nu-p}}.
	\end{equation}
	
	Let $\{v_n\}_{n\in\N}\subset \mathcal{N}_\lambda^- \subset \mathcal{N}_\lambda$ be a minimizing sequence. From Proposition \ref{prop_coerivity} we know that $\{v_n\}_{n\in\N}\subset \WH$ is bounded. So, we may assume that
	\begin{align*}
		v_n\weak v^* \quad\text{in }\WH
		\quad
		v_n\to v^* \quad \text{in }L^\nu_{\loc}(\R^N)
		\quad
		\text{and}\quad
		v_n\to v^* \quad\text{a.\,e.\,in }\R^N,
	\end{align*}
	due to \eqref{embedding}. Note that $v^*\neq 0$ by \eqref{prop_235}. Now we can use the point $t_{v^*}^2>0$ (see \eqref{prop_8}) for which we have
	\begin{align*}
		\psi_{v^*}(t_{v^*}^2)=\lambda \l\|v^*\r\|_{\nu}^{\nu}\quad\text{and}\quad \psi'_{v^*}(t_{v^*}^2)<0.
	\end{align*}
	Note that in the proof of Proposition \ref{prop_nonemptiness_and_existence} we showed that $t^2_{v^*} v^* \in \mathcal{N}_\lambda^-$.
	
	We are going to prove that $\lim_{n\to+\infty} \varrho(v_n)=\varrho(v^*)$ for a subsequence (still denoted by $v_n$). Suppose this is not true, then we have for a subsequence if necessary that
	\begin{align*}
		\Psi_\lambda\l(t^2_{v^*} v^*\r)< \lim_{n\to\infty} \Psi_\lambda\l(t^2_{v^*} v_n\r).
	\end{align*}
	Note that $\Psi_\lambda(t^2_{v^*}v_n) \leq \Psi_\lambda(v_n)$ since it is the global maximum because of $\omega_{v_n}''(1)<0$. Using this along with $t^2_{v^*} v^* \in \mathcal{N}_\lambda^-$ it follows
	\begin{align*}
		m_\lambda^- \leq \Psi_\lambda(t^2_{v^*} v^*) < m^-_\lambda,
	\end{align*}
	a contradiction. Hence, for a subsequence, we have $\lim_{n\to+\infty} \varrho(v_n)=\varrho(v^*)$ and since the integrand corresponding to the modular function $\varrho(\cdot)$ is uniformly convex, this implies $\varrho (\frac{v_n-v^*}{2})\to 0$. Then Proposition \ref{proposition_modular_properties}\textnormal{(v)} implies that $v_n \to v^*$ in $\WH$ and the continuity of $\Psi_\lambda$ gives  $\Psi_{\lambda}(v_n)\to \Psi_{\lambda}(v^*)$ and so $\Psi_{\lambda}(v^*)=m^-_{\lambda}$. 
	
	Because  of $v_n\in \mathcal{N}^-_{\lambda}$ for all $n\in \N$, we have the inequality
	\begin{align}\label{prop_15j2}
		\begin{split}
			(p+\alpha-1)\|v_n\|_{1,p}^p+(q+\alpha-1)\big(\|\nabla v_n\|_{q,\mu}^q+\|v_n\|_{q,\mu}^q\big)-\lambda (\nu+\alpha-1) \|v_n\|_{\nu}^{\nu}<0,
		\end{split}	
	\end{align}
	Passing to the limit  \eqref{prop_15j2} as $n\to+\infty$ we get
	\begin{equation*}
		(p+\alpha-1)\|v^*\|_{1,p}^p+(q+\alpha-1)\big(\|\nabla v^*\|_{q,\mu}^q+\|v^*\|_{q,\mu}^q\big)-\lambda (\nu+\alpha-1) \|v^*\|_{\nu}^{\nu}\leq 0 .
	\end{equation*}
	Taking Proposition \ref{prop_emptiness} into account, we conclude that $v^*\in \mathcal{N}^-_{\lambda}$. Since the treatment also works for $|v^*|$ instead of $v^*$, we may assume that $v^*(x)\geq 0$ for a.\,a.\,$x\in\R^N$ such that $v^*\neq 0$. Proposition \ref{prop_minimizer_negative_manifold} finally shows that $m_\lambda^->0$.
\end{proof}

Now we obtain a second weak solution of problem \eqref{problem}.

\begin{proposition}\label{prop_seond_weak_solution}
	Let hypotheses \textnormal{(H)} be satisfied and let $\lambda \in(0, \lambda^*]$. Then $v^*$ is a weak solution of problem \eqref{problem} such that $\Psi_\lambda(v^*)>0$.
\end{proposition}

\begin{proof}
	As in the proof of Proposition \ref{prop_energy_estimate} replacing $u^*$ by $v^*$ in the definition of $\eta_v$ we can prove that for every $t \in [0,\zeta_0]$ there exists $\ph(t)>0$ such that
	\begin{align*}
		\ph(t)\l(v^*+tv\r)\in \mathcal{N}_\lambda^-
		\quad\text{and}\quad
		\ph(t) \to 1 \quad\text{as } t\to 0^+.
	\end{align*}
	Applying Proposition \ref{prop_minimize_negative} gives
	\begin{align}\label{prop_200}
		m_\lambda^-=\Psi_\lambda \l(v^*\r) \leq \Psi_\lambda \l(\ph(t)\l(v^*+tv\r)\r)\quad\text{for all } t \in [0,\zeta_0].
	\end{align}
	
	Let us prove that $v^*> 0$ for a.\,a.\,$x\in\R^N$. Suppose there is a set $C$ with positive measure such that $v^*=0$ in $C$. Choosing $v\in \WH$ with $v > 0$ and let $t\in (0,\zeta_0)$, see \eqref{prop_200}, we have $(\ph(t)(v^*+tv))^{1-\alpha}>(\ph(t)v^*)^{1-\alpha}$ a.\,e.\,in $\R^N\setminus C$. From \eqref{prop_200} and since $\omega_{v^*}(1)$ is the global maximum which implies $\Psi_\lambda(v^*)=\omega_{v^*}(1) \geq \omega_{v^*}(\ph(t))=\Psi_\lambda(\ph(t)v^*)$, we then obtain 
	\begin{align*}
		0
		&\leq \frac{\Psi_\lambda(\ph(t)(v^*+tv))-\Psi_\lambda(v^*)}{t} \\
		&\leq \frac{\Psi_\lambda(\ph(t)(v^*+tv))-\Psi_\lambda(\ph(t)v^*)}{t}\\
		&=\frac{1}{p} \frac{\|\ph(t)(v^*+tv)\|_{1,p}^p-\|\ph(t)v^*\|_{1,p}^p}{t} +\frac{1}{q} \frac{\|\nabla (\ph(t)(v^*+tv))\|_{q,\mu}^q-\|\nabla (\ph(t)v^*)\|_{q,\mu}^q}{t}\\
		& \quad +\frac{1}{q} \frac{\|\ph(t)(v^*+tv)\|_{q,\mu}^q-\| \ph(t)v^*\|_{q,\mu}^q}{t}-\frac{\ph(t)^{1-\alpha}}{(1-\alpha)t^{\alpha}} \int_C \xi(x)h^{1-\alpha} \diff x\\
		&\quad - \frac{1}{1-\alpha}\int_{\R^N\setminus C} \xi(x)\frac{(\ph(t)(v^*+tv))^{1-\alpha}-(\ph(t)v^*)^{1-\alpha}}{t}\diff x
		-\frac{\lambda}{\nu}
		\frac{\|\ph(t)(v^*+tv)\|_{\nu}^{\nu}-\|\ph(t)v^*\|_{\nu}^{\nu}}{t}\\
		&<\frac{1}{p} \frac{\|\ph(t)(v^*+tv)\|_{1,p}^p-\|\ph(t)v^*\|_{1,p}^p}{t} +\frac{1}{q} \frac{\|\nabla (\ph(t)(v^*+tv))\|_{q,\mu}^q-\|\nabla (\ph(t)v^*)\|_{q,\mu}^q}{t}\\
		&\quad +\frac{1}{q} \frac{\|\ph(t)(v^*+tv)\|_{q,\mu}^q-\| \ph(t)v^*\|_{q,\mu}^q}{t}
		-\frac{\ph(t)^{1-\alpha}}{(1-\alpha)t^{\alpha}} \int_C \xi(x)h^{1-\alpha} \diff x\\
		&\quad -\frac{\lambda}{\nu}
		\frac{\|\ph(t)(v^*+tv)\|_{\nu}^{\nu}-\|\ph(t)v^*\|_{\nu}^{\nu}}{t}.
	\end{align*}
	From the considerations above we see that
	\begin{align*}
		0
		&\leq \frac{\Psi_\lambda(\ph(t)(v^*+tv))-\Psi_\lambda(\ph(t)v^*)}{t} \to -\infty \quad \text{as }t\to 0^+,
	\end{align*}
	which is a contradiction. This shows that $v^*>0$ a.\,e.\,in $\R^N$.
	
	The rest of the proof  is similar to the one of Proposition \ref{prop_first_weak_solution}. Note that \eqref{def1} and \eqref{def2} can be proven similarly using again \eqref{prop_200} and the inequality $\omega_{v^*}(1) \geq \omega_{v^*}(\ph(t))$ along with $v^*>0$. From Proposition \ref{prop_minimize_negative} we know that $\Psi_\lambda(v^*)>0$. 
\end{proof}

Finally, the proof of Theorem \ref{main_result} is now a direct consequence of Propositions \ref{prop_first_weak_solution} and \ref{prop_seond_weak_solution}.

\section*{Acknowledgment}

W. Liu was supported by the NNSF of China (Grant No. 11961030).


\end{document}